\documentclass[a4paper]{amsart}
\title{Higher matrix-tree theorems}

\newcommand{\theoremName}{Theorem}
\newcommand{\lemmaName}{Lemma}
\newcommand{\corollaryName}{Corollary}
\newcommand{\statementName}{Proposition}

\newcommand{\remarkName}{Remark}
\newcommand{\exampleName}{Example}
\newcommand{\definitionName}{Definition}
\newcommand{\problemName}{Problem}
\newcommand{\proofName}{Proof}
\renewcommand{\proofname}{\proofName}
\newcommand{\answerName}{Answer}
\newcommand{\hintName}{Hint}

\theoremstyle{plain}

\newtheorem{theorem}{\theoremName}
\newtheorem{lemma}[theorem]{\lemmaName}
\newtheorem {corollary}[theorem]{\corollaryName}

\newtheorem {Theorem}{\theoremName}

\theoremstyle{remark}

\newtheorem{remark}[theorem]{\remarkName}
\newtheorem{Remark}{\remarkName}

\theoremstyle{definition}

\numberwithin{theorem}{section}

\makeatletter
\let\@newpf\proof 
\let\proof\relax

\def \namepf[#1] {\@newpf[\proofname\ #1]}
\newenvironment{proof}{\@ifnextchar[{\namepf}{\@newpf[\proofname]}}{\qed\endtrivlist}

\newcounter{qst}
\@addtoreset{qst}{problem}

\def \Real {{\mathbb R}}

\def \lnorm#1\rnorm {\vphantom{#1}\left\|\smash{#1}\right\|}
\def \lmod#1\rmod {\vphantom{#1}\left|\smash{#1}\right|}

\newcommand \bydef {\stackrel{\mbox{\scriptsize def}}{=}}

\@ifundefined{name}{\newcommand \name[1] {\mathop{\rm #1}\nolimits}}%
{\relax}

\newcommand \eps {\varepsilon}
\renewcommand \phi {\varphi}
\renewcommand \rho {\varrho}
\makeatother

\author[Yu.Burman]{Yurii Burman}
\thanks{Independent University of Moscow and Higher School of Economics, 
Moscow,Russia} 
\address{Independent University of Moscow, 119002, 11, B.Vlassievsky per., 
Moscow, Russia} 
\email[Yu.Burman]{burman@mccme.ru} 

\author[A.Ploskonosov]{Andrey Ploskonosov} 
\email[A.Ploskonosov]{strashila@newmail.ru}

\author[A.Trofimova]{Anastasia Trofimova}
\email[A.Trofimova]{camomile252@mail.ru}

\date{}
\usepackage{emlines}
\begin{document}

 \begin{abstract}
We calculate determinants of weighted sums of reflections and of (nested) 
commutators of reflections. The results obtained generalize the Kirchhoff's 
matrix-tree theorem and the matrix-3-hypertree theorem by G.\,Massbaum and 
A.\,Vaintrob.
 \end{abstract}

\maketitle

\section*{Introduction}

A famous matrix-tree theorem first proved by Kirchhoff \cite{Kirchhoff} in 
1847 recently attracted attention of specialists in algebraic combinatorics 
(see e.g.\ the paper \cite{Chaiken} containing a good review of other 
results in the field as well, and also some generalizations in \cite{BSh}). 
The theorem in its classical form expresses the principal minor of some $n 
\times n$-matrix (see $L_w$ in Section \ref{SSec:AnK1} below) via summation 
over the set of trees on $n$ numbered vertices. The matrix involved is the 
matrix of a weighted sum of the operators $(I - s)$ where $s$ runs through 
all the reflections in the Coxeter group $A_{n-1}$. We prove an analog of 
Kirchhoff's formula for {\em any} system of reflections in a Euclidean 
space; instead of trees, our result involves summation over the set of 
bases made out of vectors normal to reflection hyperplanes.

In 2002 G.\,Massbaum and A.\,Vaintrob found a beautiful extension of the 
Kirchhoff's theorem. Their result expresses the Pfaffian of a principal 
minor of some skew-symmetric $(2m+1) \times (2m+1)$-matrix via summation 
over a set of $3$-trees (contractible topological spaces made by gluing 
triangles by their vertices) with $m$ edges. The matrix (see $T$ in Section 
\ref{SSec:MV} below) is the weighted sum of commutants $[s_1, s_2]$ where 
$s_1$ and $s_2$ run through the $A_{n-1}$ set of reflections. Again, we 
generalize this theorem to an arbitrary set of reflections, allowing also 
nested commutators of the form $[s_1, [s_2, \dots, s_k] \dots ]$ for any 
$k$. This is the main result of the paper, Theorem \ref{Th:GenDet} and its 
reformulation, Theorem \ref{Th:GenDet}'. The answer is given via summation 
over the set of ``discrete one-dimensional oriented manifold with 
boundary'', which are just directed graphs made up of chains and cycles; 
the weight of the graph is obtained by a sort of discrete path integration.

\section{General theorems}\label{Sec:General}

Let $V$ be a $n$-dimensional Euclidean space with an orthonormal basis 
$f_1, \dots, f_n$, and let $e_1, \dots, e_N \in V$ be vectors of unit 
length. Denote by $s_i: V \to V$ the reflection in the hyperplane normal to 
$e_i$: $s_i(v) = v - 2(e_i,v)e_i$. 

Denote by $Q_e: V \to V$, where $e \bydef (e_1, \dots, e_k)$, a rank $1$ 
linear operator given by the formula
 %*
 \begin{equation*}
Q_e(v) = (v,e_1)(e_1,e_2) \dots (e_{k-1},e_k) e_k
 \end{equation*}
 %*
$e \mapsto Q_e$ is a $\name{End}(V)$-valued quadratic function of its 
arguments. Fix now $e_1, \dots, e_k$, and for any permutation $\sigma \in 
S_k$ consider an operator $Q_{\sigma(e)}: V \to V$ where $\sigma(e) \bydef 
(e_{\sigma(1)}, \dots, e_{\sigma(k)})$. Extend the correspondence $\sigma 
\mapsto Q_{\sigma(e)}$ by linearity to a quadratic map $Q: \Real[S_k] \to 
\name{End}(V)$ from the group algebra of the symmetric group $S_k$ to the 
algebra of linear operators $V \to V$: if $x = \sum_{\sigma \in S_k} 
c_\sigma \sigma$ (where $c_\sigma \in \Real$) then $Q(x) \bydef 
\sum_{\sigma \in S_k} c_\sigma Q_{\sigma(e)}$.

An easy induction proves the following
 \begin{lemma} \label{Lm:MultiComm}
$[s_k, [s_{k-1}, [\dots, [s_2, s_1]\dots] = Q(x_k)$ with $x_k = 
-2^k(\tau_k-1)(\tau_{k-1}-1) \dots (\tau_2-1)$ where $\tau_\ell \bydef (12 
\dots \ell)$, a permutation mapping $1 \mapsto 2 \mapsto \dots \mapsto \ell 
\mapsto 1$ and leaving all $i > \ell$ fixed. 
 \end{lemma}

Define integers $a_k(\sigma)$ by the equation $x_k = \sum_{\sigma \in S_k} 
a_k(\sigma) \sigma$.

Introduce now, for all $1 \le i_1, \dots, i_k \le N$, weights $w_{i_1, 
\dots, i_k}$, which are elements of a commutative associative algebra. 
Consider the operator
 %*
 \begin{equation}\label{Eq:DefP}
P_w^{(k)} \bydef \sum_{i_1, \dots, i_k = 1}^N w_{i_1, \dots, i_k} [s_{i_k}, 
[s_{i_{k-1}},\dots, [s_{i_2},s_{i_1}]\dots]
 \end{equation}
 %*
By Lemma \ref{Lm:MultiComm} 
 %*
 \begin{equation}\label{Eq:Aw}
P_w^{(k)} = \sum_{i_1,\dots,i_k = 1}^N \sum_{\sigma 
\in S_k} a_k(\sigma) Q(e_{i_{\sigma(1)}}, \dots, e_{i_{\sigma(k)}}) =
\sum_{j_1,\dots,j_k = 1}^N u_{j_1, \dots, j_k} Q(e_{j_1}, \dots, e_{j_k})
 \end{equation}
 %*
where $u_{j_1, \dots, j_k} \bydef \sum_{\sigma \in S_k} a_k(\sigma^{-1}) 
w_{j_{\sigma(1)}, \dots, j_{\sigma(k)}}$.

 \begin{corollary}[of Lemma \ref{Lm:MultiComm}]\label{Cr:Reverse}
$u_{j_k, \dots, j_1} = (-1)^{k-1} u_{j_1, \dots, j_k}$
 \end{corollary}

 \begin{proof}
According to Lemma \ref{Lm:MultiComm} $u_{j_1, \dots, j_k}$ changes its 
sign if a permutation $\tau_\ell$ with any $2 \le \ell \le k$ is applied to 
its arguments. Note now that $\tau_2 \dots \tau_k$ is a permutation of $1, 
\dots, k$ exchanging $1 \leftrightarrow k$, $2 \leftrightarrow k-1$, etc.
 \end{proof}

 \begin{theorem} \label{Th:GenDet}
 %*
 \begin{align}
\det P_w^{(k)} = \sum^{(N, \dots, N)}_{\begin{array}{c} \scriptstyle 
j^{(1)}, \dots, j^{(n)} = (1,\dots,1)\\ \scriptstyle j^{(1)}_k < \dots < 
j^{(n)}_k \end{array}} &u_{j^{(1)}} \dots u_{j^{(n)}} \prod_{s=1}^{k-1} 
\prod_{t=1}^n (e_{j_s^{(t)}}, e_{j_{s+1}^{(t)}}) \label{Eq:GenDet}\\
&\times \name{vol}(e_{j_1^{(1)}}, \dots, 
e_{j_1^{(n)}})\name{vol}(e_{j_k^{(1)}}, \dots, e_{j_k^{(n)}}).\nonumber
 \end{align}
 %*
Here $j^{(s)} \bydef (j^{(s)}_1, \dots, j^{(s)}_k)$ is a multi-index, and 
$\name{vol}(b_1, \dots, b_n)$ stands for an $n$-dimensional volume of the 
parallelepiped spanned by the vectors $b_1, \dots, b_n \in V$.
 \end{theorem}

 \begin{proof}
Let $f_1, \dots, f_n$ be an orthonormal basis in $V$. Then it follows from 
\eqref{Eq:Aw} that 
 %*
 \begin{align*}
P_w^{(k)}(f_1) \wedge &\dots \wedge P_w^{(k)}(f_n) \\
&= \sum_{j^{(1)}, \dots, j^{(n)} 
= (1, \dots, 1)}^{(N,\dots,N)} u_{j^{(1)}} \dots u_{j^{(n)}} 
Q_{(e_{j_1^{(1)}}, \dots, e_{j_k^{(1)}})}(f_1) \wedge \dots 
\wedge Q_{(e_{j_1^{(n)}}, \dots, e_{j_k^{(n)}})}(f_n)\\ 
&= \sum_{j^{(1)}, \dots, j^{(n)} = (1, \dots, 1)}^{(N,\dots,N)} u_{j^{(1)}} 
\dots u_{j^{(n)}} \prod_{p=1}^n (e_{j_1^{(p)}}, f_p) \prod_{s=1}^{k-1} 
\prod_{t=1}^n (e_{j_s^{(t)}}, e_{j_{s+1}^{(t)}})\\
&\hphantom{= \sum_{j^{(1)}, \dots, j^{(n)} = (1, \dots, 1)}^{(N,\dots,N)} 
u_{j^{(1)}} \dots u_{j^{(n)}} \prod_{p=1}^n (e_{j_1^{(p)}}, f_p) 
\prod_{s=1}^{k-1} \prod_{t=1}^n}
\times e_{j_k^{(1)}} \wedge \dots \wedge e_{j_k^{(n)}}
 \end{align*}
 %*
The wedge product at the end changes its sign if two multi-indices 
$j^{(a)}$ and $j^{(b)}$ are exchanged. Additionally, all $j_k^{(i)}$ should 
be distinct, else the term is zero. Thus we can restrict summation to the 
multi-indices such that $j^{(1)}_k < \dots < j^{(n)}_k$, and then make an 
additional summation over the set of permutations $\sigma$ of $n$ points:
 %*
 \begin{align*}
P_w^{(k)}(f_1) \wedge &\dots \wedge P_w^{(k)}(f_n) \\
&= \sum^{(N, \dots, N)}_{\begin{array}{c} \scriptstyle j^{(1)}, 
\dots, j^{(n)} = (1,\dots,1)\\ \scriptstyle j^{(1)}_k < \dots < j^{(n)}_k 
\end{array}} u_{j^{(1)}} \dots u_{j^{(n)}} \prod_{s=1}^{k-1} \prod_{t=1}^n 
(e_{j_s^{(t)}}, e_{j_{s+1}^{(t)}}) \\
&\times \sum_{\sigma \in S_n} (-1)^\sigma \prod_{p=1}^n (e_{j_1^{(p)}}, 
f_{\sigma(p)}) e_{j_k^{(1)}} \wedge \dots \wedge e_{j_k^{(n)}}\\ 
&= \sum^{(N, \dots, N)}_{\begin{array}{c} \scriptstyle j^{(1)}, \dots, 
j^{(n)} = (1,\dots,1)\\ \scriptstyle j^{(1)}_k < \dots < j^{(n)}_k 
\end{array}} u_{j^{(1)}} \dots u_{j^{(n)}} \prod_{s=1}^{k-1} \prod_{t=1}^n 
(e_{j_s^{(t)}}, e_{j_{s+1}^{(t)}}) \\ 
&\times \name{vol}(e_{j_1^{(1)}}, \dots, 
e_{j_1^{(n)}}) \name{vol}(e_{j_k^{(1)}}, \dots, e_{j_k^{(n)}}) \times f_1 
\wedge \dots \wedge f_n.
 \end{align*}
 %*
The coefficient at $f_1 \wedge \dots \wedge f_n$ is $\det P_w^{(k)}$.
 \end{proof}

Theorem \ref{Th:GenDet} admits a beautiful reformulation. For every set 
$j^{(1)}, \dots, j^{(n)}$ of multi-indices in the sum draw a graph with 
the vertices $1, \dots, N$ and oriented edges joining $j_1^{(i)}$ with 
$j_k^{(i)}$ for all $i = 1, \dots, n$. The volume in the formula may be 
nonzero only if all its arguments are distinct; so in the graph obtained 
for every vertex there is at most one outgoing edge and at most one 
incoming edge. This means that every connected component of the graph is 
either an oriented cycle or an oriented chain --- thus, the graph is a 
``discrete one-dimensional oriented manifold with boundary'' (abbreviated 
as {\em DOOMB} below).

For every edge $\eps_i = (j_1^{(i)}, j_k^{(i)})$ of the graph consider a 
path (a sequence of vertices) $j^{(i)} = (j_1^{(i)}, j_2^{(i)}, \dots, 
j_k^{(i)})$. This path has a weight $u_{j^{(i)}} (e_{j_1^{(i)}}, 
e_{j_2^{(i)}}) \dots (e_{j_{k-1}^{(i)}}, e_{j_k^{(i)}})$; call the 
$k$-weight of the edge $\eps_i$ the sum of weights of all the paths of 
length $k$ joining its endpoints. The $k$-weight of the graph $\Gamma$ is 
the product of the $k$-weights of its edges. Then Theorem \ref{Th:GenDet} 
is equivalent to

{\def \theoremName {Theorem \ref{Th:GenDet}'}
 \begin{Theorem}
The determinant of the operator $P_w^{(k)}$ is equal to the sum of 
$k$-weights of all the DOOMBs having $n$ edges with the vertices $1, \dots, 
N$, each weight multiplied by $\name{vol}(e_{p_1}, \dots, e_{p_n}) 
\name{vol}(e_{q_1}, \dots, e_{q_n})$ where $p_i$ and $q_i$ are a starting 
vertex and a final vertex, respectively, of the $i$-th edge of the graph, 
$1 \le i \le n$.
 \end{Theorem}
}

 \begin{remark} \label{Rm:IndepNum}
To formulate Theorem \ref{Th:GenDet}' it is necessary to number the edges 
for every DOOMB involved, but the value of the corresponding term is 
independent of the numbering.
 \end{remark}

Formulations of Lemma \ref{Lm:MultiComm} and Theorem \ref{Th:GenDet} assume 
that $k \ge 2$. The propositions remain valid, nevertheless, together with 
their proofs, if $k = 1$, if one defines $u_i \bydef w_i$ for all $i$ and 
$P_w^{(1)} \bydef \sum_{i=1}^N w_i (I - s_i)$ ($I$ means the identity 
operator). More precisely, the following is true:

{\def \theoremName {Theorem \ref{Th:GenDet} for $k=1$}
 \begin{Theorem}
 %*
 \begin{equation} \label{Eq:NoCommut}
\det P_w^{(1)} = 2^n \sum_{1 \le i_1 < \dots < i_n \le N} w_{i_1} \dots 
w_{i_n} \name{vol}^2(e_{i_1}, \dots, e_{i_n}).
 \end{equation}
 %*
 \end{Theorem}
}

A reflection is an orthogonal operator, and it is an involution, so it is 
a symmetric operator. Therefore, a nested commutator of $k$ reflections is 
symmetric for $k$ odd and skew-symmetric for $k$ even, and the same is 
true for the operator $P_w^{(k)}$. Suppose that $k$ is even. In this case
$\det P_w^{(k)}$ may be nonzero only if $n \bydef \dim V$ is even, and 
$\det P_w^{(k)} = \name{Pf}^2 P_w^{(k)}$ where $\name{Pf}$ means the 
Pfaffian. 

Consider now the term in Theorem \ref{Th:GenDet}' corresponding to the 
DOOMB $\Gamma$. Take a connected component of $\Gamma$ and reverse the 
orientation of all the edges in it obtaining a new DOOMB $\Gamma'$. It 
follows from Corollary \ref{Cr:Reverse} that for $k$ even the weight of an 
edge changes its sign if the orientation of the edge is reversed. Therefore 
the contribution to the sum of $\Gamma'$ equals the contribution of 
$\Gamma$ multiplied by $(-1)^\ell$ where $\ell$ is the size of the 
component. If $\ell$ is odd, the contributions of $\Gamma$ and $\Gamma'$ 
cancel --- therefore, for $k$ even only DOOMBs with connected components of 
even size enter the sum. 

Recall that a directed graph on the vertex set $1, \dots, N$ is called a 
directed partial pair matching if no two its edges have a common vertex. 
Every DOOMB with $n$ edges all of whose connected components have even 
size can be uniquely decomposed into a union of two directed partial pair 
matchings of $n/2$ edges each. Vice versa, a union of two directed partial 
pair matchings is a DOOMB with connected components of even size. Thus for 
$k$ even the following statement is equivalent to Theorem \ref{Th:GenDet}:

{\def \theoremName {Theorem \ref{Th:GenDet} for $k$ even}
 \begin{Theorem}
If $k$ is even then a Pfaffian of $P_w^{(k)}$ is the sum of $k$-weights of 
all the directed partial pair matchings with $n/2$ edges and vertices $1, 
\dots, N$, multiplied by the volume of the parallelepiped spanned by the 
vectors $e_p$ where $p$ runs through all the vertices of the pair matching.
 \end{Theorem}
}

\section{The $A_n$ case}\label{Sec:An}

In this section we consider the system of reflections of the Coxeter 
group $A_n$. This means $V = \{(x_0, \dots, x_n) \mid x_1 + 
\dots + x_n = 0\} \subset \Real^{n+1}$, $N = n(n+1)/2$, and the vectors 
$e_p$ are 
 %*
 \begin{equation}\label{Eq:DefEij}
e_{ij} \bydef (f_i - f_j)/\sqrt{2}, \quad 0 \le i < j \le n,
 \end{equation}
 %*
where $f_0, \dots, f_n$ is the standard orthonormal basis in $\Real^{n+1}$. 
It will be convenient to use notation \eqref{Eq:DefEij} also if $i > j$, so 
that $e_{ji} = -e_{ij}$. Consider a set of points numbered $0, \dots, n$; a 
vector $e_{ij}$ will be drawn then as an arrow joining points $i$ and $j$. 
For any set $u = \{e_{i_1 j_1}, \dots, e_{i_k j_k}\}$ we will denote by 
$G(u)$ a graph with vertices $0, \dots, n$ and directed edges $(i_1,j_1), 
\dots, (i_k,j_k)$.

The scalar product in $\Real^n$ is standard, so the scalar product of 
$e_{ij}$ is given by
 %*
 \begin{equation}\label{Eq:ScPr}
(e_{i_1,j_1}, e_{i_2,j_2}) =  \begin{cases}
$1$, &\text{if {
\setlength{\unitlength}{0.075in}
\begin{picture}(5.00,3.28)
\special{em:linewidth 0.014in}
\put(5.00,1.65){\special{em:moveto}}
\put(4.03,1.11){\special{em:lineto}}
\put(5.00,0.78){\special{em:moveto}}
\put(5.00,1.65){\special{em:lineto}}
\put(5.00,1.78){\special{em:moveto}}
\put(5.00,2.61){\special{em:lineto}}
\put(4.96,1.78){\special{em:moveto}}
\put(3.96,2.28){\special{em:lineto}}
\put(5.00,1.78){\special{em:moveto}}
\put(4.78,1.33){\special{em:lineto}}
\put(4.46,0.88){\special{em:lineto}}
\put(4.08,0.53){\special{em:lineto}}
\put(3.76,0.33){\special{em:lineto}}
\put(3.45,0.18){\special{em:lineto}}
\put(2.90,0.03){\special{em:lineto}}
\put(2.48,0.00){\special{em:lineto}}
\put(1.90,0.06){\special{em:lineto}}
\put(1.60,0.15){\special{em:lineto}}
\put(1.30,0.28){\special{em:lineto}}
\put(1.13,0.36){\special{em:lineto}}
\put(0.98,0.46){\special{em:lineto}}
\put(0.85,0.56){\special{em:lineto}}
\put(0.75,0.65){\special{em:lineto}}
\put(0.65,0.73){\special{em:lineto}}
\put(0.51,0.86){\special{em:lineto}}
\put(0.41,1.00){\special{em:lineto}}
\put(0.35,1.08){\special{em:lineto}}
\put(0.28,1.16){\special{em:lineto}}
\put(0.25,1.23){\special{em:lineto}}
\put(0.20,1.30){\special{em:lineto}}
\put(0.16,1.36){\special{em:lineto}}
\put(0.13,1.43){\special{em:lineto}}
\put(0.10,1.50){\special{em:lineto}}
\put(0.06,1.56){\special{em:lineto}}
\put(0.00,1.58){\special{em:lineto}}
\put(0.00,1.58){\special{em:moveto}}
\put(0.10,1.78){\special{em:lineto}}
\put(0.20,1.98){\special{em:lineto}}
\put(0.33,2.18){\special{em:lineto}}
\put(0.50,2.38){\special{em:lineto}}
\put(0.60,2.48){\special{em:lineto}}
\put(0.70,2.58){\special{em:lineto}}
\put(0.76,2.65){\special{em:lineto}}
\put(0.85,2.71){\special{em:lineto}}
\put(0.95,2.78){\special{em:lineto}}
\put(1.05,2.85){\special{em:lineto}}
\put(1.15,2.91){\special{em:lineto}}
\put(1.26,2.98){\special{em:lineto}}
\put(1.40,3.05){\special{em:lineto}}
\put(1.56,3.11){\special{em:lineto}}
\put(1.76,3.18){\special{em:lineto}}
\put(2.05,3.25){\special{em:lineto}}
\put(2.51,3.28){\special{em:lineto}}
\put(3.13,3.21){\special{em:lineto}}
\put(3.36,3.15){\special{em:lineto}}
\put(3.55,3.08){\special{em:lineto}}
\put(3.73,3.00){\special{em:lineto}}
\put(3.88,2.91){\special{em:lineto}}
\put(3.98,2.85){\special{em:lineto}}
\put(4.08,2.78){\special{em:lineto}}
\put(4.20,2.70){\special{em:lineto}}
\put(4.28,2.63){\special{em:lineto}}
\put(4.45,2.46){\special{em:lineto}}
\put(4.60,2.30){\special{em:lineto}}
\put(4.73,2.13){\special{em:lineto}}
\put(4.83,1.96){\special{em:lineto}}
\put(4.93,1.80){\special{em:lineto}}
\put(5.00,1.68){\special{em:lineto}}
\end{picture}
}};\\
$-1$, &\text{if {
\setlength{\unitlength}{0.075in}
\begin{picture}(5.00,3.28)
\special{em:linewidth 0.014in}
\put(0.00,1.61){\special{em:moveto}}
\put(0.76,1.11){\special{em:lineto}}
\put(0.00,0.81){\special{em:moveto}}
\put(0.00,1.61){\special{em:lineto}}
\put(5.00,1.78){\special{em:moveto}}
\put(5.00,2.61){\special{em:lineto}}
\put(4.96,1.78){\special{em:moveto}}
\put(3.96,2.28){\special{em:lineto}}
\put(5.00,1.78){\special{em:moveto}}
\put(4.78,1.33){\special{em:lineto}}
\put(4.46,0.88){\special{em:lineto}}
\put(4.08,0.53){\special{em:lineto}}
\put(3.76,0.33){\special{em:lineto}}
\put(3.45,0.18){\special{em:lineto}}
\put(2.90,0.03){\special{em:lineto}}
\put(2.48,0.00){\special{em:lineto}}
\put(1.90,0.06){\special{em:lineto}}
\put(1.60,0.15){\special{em:lineto}}
\put(1.30,0.28){\special{em:lineto}}
\put(1.13,0.36){\special{em:lineto}}
\put(0.98,0.46){\special{em:lineto}}
\put(0.85,0.56){\special{em:lineto}}
\put(0.75,0.65){\special{em:lineto}}
\put(0.65,0.73){\special{em:lineto}}
\put(0.51,0.86){\special{em:lineto}}
\put(0.41,1.00){\special{em:lineto}}
\put(0.35,1.08){\special{em:lineto}}
\put(0.28,1.16){\special{em:lineto}}
\put(0.25,1.23){\special{em:lineto}}
\put(0.20,1.30){\special{em:lineto}}
\put(0.16,1.36){\special{em:lineto}}
\put(0.13,1.43){\special{em:lineto}}
\put(0.10,1.50){\special{em:lineto}}
\put(0.06,1.56){\special{em:lineto}}
\put(0.00,1.58){\special{em:lineto}}
\put(0.00,1.58){\special{em:moveto}}
\put(0.10,1.78){\special{em:lineto}}
\put(0.20,1.98){\special{em:lineto}}
\put(0.33,2.18){\special{em:lineto}}
\put(0.50,2.38){\special{em:lineto}}
\put(0.60,2.48){\special{em:lineto}}
\put(0.70,2.58){\special{em:lineto}}
\put(0.76,2.65){\special{em:lineto}}
\put(0.85,2.71){\special{em:lineto}}
\put(0.95,2.78){\special{em:lineto}}
\put(1.05,2.85){\special{em:lineto}}
\put(1.15,2.91){\special{em:lineto}}
\put(1.26,2.98){\special{em:lineto}}
\put(1.40,3.05){\special{em:lineto}}
\put(1.56,3.11){\special{em:lineto}}
\put(1.76,3.18){\special{em:lineto}}
\put(2.05,3.25){\special{em:lineto}}
\put(2.51,3.28){\special{em:lineto}}
\put(3.13,3.21){\special{em:lineto}}
\put(3.36,3.15){\special{em:lineto}}
\put(3.55,3.08){\special{em:lineto}}
\put(3.73,3.00){\special{em:lineto}}
\put(3.88,2.91){\special{em:lineto}}
\put(3.98,2.85){\special{em:lineto}}
\put(4.08,2.78){\special{em:lineto}}
\put(4.20,2.70){\special{em:lineto}}
\put(4.28,2.63){\special{em:lineto}}
\put(4.45,2.46){\special{em:lineto}}
\put(4.60,2.30){\special{em:lineto}}
\put(4.73,2.13){\special{em:lineto}}
\put(4.83,1.96){\special{em:lineto}}
\put(4.93,1.80){\special{em:lineto}}
\put(5.00,1.68){\special{em:lineto}}
\end{picture}
}};\\
$1/2$, &\text{if {
\setlength{\unitlength}{0.075in}
\begin{picture}(3.33,3.33)
\special{em:linewidth 0.014in}
\put(3.33,3.33){\special{em:moveto}}
\put(2.96,2.73){\special{em:lineto}}
\put(2.53,3.33){\special{em:moveto}}
\put(3.33,3.33){\special{em:lineto}}
\put(3.33,0.00){\special{em:moveto}}
\put(2.63,0.00){\special{em:lineto}}
\put(3.03,0.60){\special{em:moveto}}
\put(3.33,0.00){\special{em:lineto}}
\put(0.00,1.66){\special{em:moveto}}
\put(3.33,0.00){\special{em:lineto}}
\put(3.33,3.33){\special{em:moveto}}
\put(0.00,1.66){\special{em:lineto}}
\end{picture}
} or {
\setlength{\unitlength}{0.075in}
\begin{picture}(3.33,3.36)
\special{em:linewidth 0.014in}
\put(3.33,1.70){\special{em:moveto}}
\put(3.00,2.26){\special{em:lineto}}
\put(3.33,1.70){\special{em:moveto}}
\put(2.90,1.13){\special{em:lineto}}
\put(2.53,1.70){\special{em:moveto}}
\put(3.33,1.70){\special{em:lineto}}
\put(3.33,1.70){\special{em:moveto}}
\put(0.00,0.00){\special{em:lineto}}
\put(0.00,3.36){\special{em:moveto}}
\put(3.33,1.70){\special{em:lineto}}
\end{picture}
}};\\
$-1/2$, &\text{if {
\setlength{\unitlength}{0.075in}
\begin{picture}(6.66,0.93)
\special{em:linewidth 0.014in}
\put(6.60,0.46){\special{em:moveto}}
\put(5.96,0.00){\special{em:lineto}}
\put(5.86,0.93){\special{em:moveto}}
\put(6.60,0.46){\special{em:lineto}}
\put(3.26,0.46){\special{em:moveto}}
\put(2.63,0.06){\special{em:lineto}}
\put(2.56,0.86){\special{em:moveto}}
\put(3.26,0.46){\special{em:lineto}}
\put(0.00,0.46){\special{em:moveto}}
\put(6.66,0.46){\special{em:lineto}}
\end{picture}
}}; \\
$0$, &\text{in all the other cases, that is, if {
\setlength{\unitlength}{0.075in}
\begin{picture}(3.33,2.40)
\special{em:linewidth 0.014in}
\put(3.30,0.43){\special{em:moveto}}
\put(2.86,0.00){\special{em:lineto}}
\put(2.80,0.83){\special{em:moveto}}
\put(3.30,0.43){\special{em:lineto}}
\put(3.33,2.06){\special{em:moveto}}
\put(2.76,1.70){\special{em:lineto}}
\put(2.70,2.40){\special{em:moveto}}
\put(3.33,2.06){\special{em:lineto}}
\put(3.33,0.40){\special{em:moveto}}
\put(0.00,0.40){\special{em:lineto}}
\put(0.00,2.06){\special{em:moveto}}
\put(3.33,2.06){\special{em:lineto}}
\end{picture}
}}
 \end{cases}
 \end{equation}
 %*

 \begin{lemma}\label{Lm:DetAn}
The determinant $\name{vol}(e_{i_1,j_1}, \dots, e_{i_n,j_n})$ is 
zero unless the graph \linebreak $G(e_{i_1,j_1}, \dots, e_{i_n,j_n})$ is a 
tree. If it is a tree then the volume is equal to $\pm\sqrt{(n+1)/2^n}$.
 \end{lemma}

 \begin{proof}
Denote $u = (e_{i_1,j_1}, \dots, e_{i_n,j_n})$. If $G(u)$ is not a tree, 
then it contains at least one cycle. Without loss of generality the cycle 
is formed by the vertices $i_1, \dots, i_k$ and the edges $(i_1, i_2), 
\dots, (i_k,i_1)$. Then $e_{i_1 i_2} + \dots + e_{i_k i_1} = 0$, and $u$ is 
linearly dependent, so that the volume is zero.

Let now $G(u)$ be a tree; the vertex $0$ will serve as its root. Changing 
a direction of an edge $(i,j) \in G(u)$ means replacement $e_{ij} \mapsto 
-e_{ij}$. So, preserving $\name{vol}(u)$ up to a sign one may assume that 
all the edges of $G(u)$ are directed from the root outwards. 

Let $p$ be a hanging vertex of $G(u)$ with the parent $q$ and the 
grandparent $r$. Replacement of the edge $(q,p)$ by $(r,p)$ means 
replacement of the vector $\pm e_{qp} \in u$ by $e_{rp} = e_{qp} + e_{rq}$, 
obtaining a new set $u'$. Since $e_{rq} \in u$, one has $\name{vol}(u') = 
\name{vol}(u)$. Doing like this several times one obtains a tree where the 
root $0$ is connected with the $n$ other vertices $1, \dots, n$. The square 
$\gamma_n \bydef \name{vol}^2(u)$ is the Gram determinant; by 
\eqref{Eq:ScPr} for such tree $\gamma_n$ is the determinant of the $n 
\times n$-matrix with $1$s in the main diagonal, $1/2$'s one line above and 
one line below it, and zeros in the other places. Expanding the determinant 
along the first row one obtains $\gamma_n = \gamma_{n-1} - \frac14 \cdot 
\gamma_{n-2}$, hence $\gamma_n = (n+1)/2^n$ by induction.
 \end{proof}

 \begin{Remark}
The procedure described in the proof allows to determine the sign of the 
volume as well. Namely, take the vertex number $0$ as a root of the tree 
and relate to every edge of the tree its outer endpoint. This puts the 
edges into a one-to-one correspondence with the vertices $1, \dots, n$. The 
ordering of the vectors $e_{i_1 j_1}, \dots, e_{i_n j_n}$ determines the 
ordering of the vertices $1, \dots, n$, that is, a permutation of $1, 
\dots, n$. A direct computation shows that the sign of $\name{vol}(u)$ is 
equal to $\eps_1 \eps_2$ where $\eps_1$ is the sign of this permutation and 
$\eps_2$ is the parity of the number of edges in $G(u)$ looking {\em 
towards} the root.
 \end{Remark}

Here are two particular cases of Theorem \ref{Th:GenDet} for the $A_n$ 
system of roots:

\subsection{$k=1$: a matrix-tree theorem}\label{SSec:AnK1}

In view of Lemma \ref{Lm:DetAn}, Theorem \ref{Th:GenDet}' (for $k=1$) gives 
for the case considered: $\det P_w^{(1)} = (n+1) \sum w_{i_1 j_1} \dots 
w_{i_n j_n}$ where the sum is taken over all sets $\{(i_1, j_1), \dots, 
(i_n,j_n)\}$ of pairs such that the corresponding graph is a tree. 

Kirchhoff's (or Laplacian) matrix $L_w$ is defined as the $(n+1) \times 
(n+1)$-matrix with the entries $(L_w)_{ij} = -w_{ij}$ when $i \ne j$, and 
$(L_w)_{ii} = \sum_{j \ne i} w_{ij}$. $L_w$ is the matrix of the operator 
$-P_w^{(1)}$ in the standard basis $f_0, \dots, f_n \in \Real^{n+1}$. The 
matrix $L_w$ is degenerate: $L_w(f_0 + \dots + f_n) = 0$; therefore $L_w = 
L_w R$ where $R$ is the orthogonal projection to the subspace $V \subset 
\Real^{n+1}$; explicitly $R f_i = f_i - \frac{1}{n+1}(f_0 + \dots + f_n)$.

For any $i = 0, \dots, n$ define the operator $M_i$ by $M_i(f_i) = 0$ and 
$M_i(f_j) = f_j$ for $j \ne i$. $M_i$ is the orthogonal projection to the 
subspace $W_i \bydef \langle f_0, \dots, \widehat{f_i}, \dots, f_n \rangle 
\subset \Real^{n+1}$. Then 
$\left.M_i\right|_{V \to M_i} P_w^{(1)} \left.R\right|_{W_i \to V}$ is the 
operator $W_i \to W_i$ whose matrix in the basis $f_0, \dots, 
\widehat{f_i}, \dots, f_n$ is the $n \times n$-submatrix of $L_w$ 
obtained by deletion of the $i$-th row and the $i$-th column. Hence, the 
determinant of the submatrix (the principal minor of $L_w$) is equal to 
$\det \left.M_k R\right|_{W_k} \cdot \det P_w^{(1)}$. The matrix of the 
operator $\left.M_k R\right|_{W_k}$ in the same basis is $I - 
\frac{1}{n+1}U$ where the matrix $U$ is defined as $U_{ij} = 1$ for all $i, 
j$. So, the rank of $U$ is $1$, and $Uf = nf$ where $f = f_0 + \dots + 
\widehat{f_k} + \dots + f_n$. Therefore, the characteristic polynomial of 
$\frac{1}{n+1}U$ is $\chi(t) = t^{n-1}(t-\frac{n}{n+1})$, and hence $\det 
(I - \frac{1}{n+1}U) = \chi(1) = 1/(n+1)$. This proves the classical {\em 
matrix-tree theorem} (see \cite{Kirchhoff} and also \cite{Chaiken} for 
another proof):

 \begin{theorem}[matrix-tree theorem, \cite{Kirchhoff}]
A principal $n$-minor of the Kirchhoff's matrix is equal to the sum $\sum 
w_{i_1 j_1} \dots w_{i_n j_n}$ taken over all sets $\{(i_1, j_1), \dots, 
(i_n,j_n)\}$ representing edges of a tree with vertices $0, \dots, n$.
 \end{theorem}

\subsection{$k=2$: Massbaum--Vaintrob theorem}\label{SSec:MV}

Consider now the operator $P_w^{(2)}$ for the $A_n$ system of roots $e_{ij} 
= (f_i - f_j)/\sqrt{2}$. The Pfaffian of $P_w^{(2)}$ may be nonzero only if 
$\dim V = n$ is even; hence $n \bydef 2m$.

Let $s_{pq}$ be a reflection in the hyperplane normal to $e_{pq}$, $p \ne 
q$; thus $s_{qp} = s_{pq}$. One has $[s_{ij}, s_{kl}] = 0$ if the $\{i,j\} 
\cap \{k,l\} = \emptyset$ or $\{i,j\} = \{k,l\}$, so the only nonzero terms 
in \eqref{Eq:DefP} are $w_{ij,jk}[s_{ij},s_{jk}] \bydef w_{ijk} M_{ijk}$. 
An immediate computation shows that $M_{ijk} = \mu_{ijk} - \mu_{ikj}$ where 
$\mu_{ijk}: \Real^{n+1} \to \Real^{n+1}$ is a $3$-cyclic permutation of the 
coordinates: it maps $f_i \mapsto f_j \mapsto f_k \mapsto f_i$ and leaves 
all the other $f_p$ invariant. Hence, the operators $M_{ijk}$ are fully 
skew-symmetric (change the sign when any two indices are exchanged), and one 
has $P_w^{(2)} = \sum_{1 \le i < j < k \le N} \lambda_{ijk} M_{ijk}$ where 
$\lambda_{ijk}$ is the result of alternation of the $w_{ijk}$: 
$\lambda_{ijk} = w_{ijk} - w_{ikj} - w_{jik} - w_{kji} + w_{jki} + 
w_{kij}$. The coefficient $\lambda_{ijk}$ is fully skew-symmetric, too. 

The matrix $T = (t_{pq})$, $0 \le p,q \le n$, of the operator $P_w^{(2)}$ 
in the standard basis $f_0, \dots, f_n$ is $t_{pq} = \sum_{r=1}^N 
\lambda_{pqr}$. The matrix is degenerate because it is skew-symmetric and 
$n+1 = 2m+1$ is odd; reasoning like in Section \ref{SSec:AnK1} one proves 
that the Pfaffian of any principal $n$-minor of $T$ is equal to 
$\sqrt{n+1}$ times the Pfaffian of the operator $P_w^{(2)}$ acting in the 
subspace $V \subset \Real^{n+1}$.

The right-hand side of \eqref{Eq:GenDet} contains the coefficients 
$u_{ij,kl} = w_{ij,kl} - w_{kl,ij}$, which are all zeros except, possibly, 
$u_{ij,jk} = w_{ij,jk} - w_{jk,ij} = w_{ijk} - w_{kji} \bydef u_{ijk}$; 
this implies $\lambda_{ijk} = u_{ijk} + u_{jki} + u_{kij}$. 

Theorem \ref{Th:GenDet} for $k=2$ and the $A_n$ system of roots looks as 
follows:
 %*
 \begin{equation*}
\name{Pf} P_w^{(2)} = \sum 
\prod_{p=1}^m u_{i_p j_p k_p} (e_{i_p j_p}, e_{j_p k_p}) \name{vol}(e_{i_1 
j_1}, e_{j_1 k_1}, \dots, e_{i_m j_m}, e_{j_m k_m});
 \end{equation*}
 %*
the sum is taken over the set of all $m$-edge partial pair matchings 
$\bigl((i_1,j_1), (j_1,k_1)\bigr), \linebreak\dots, \bigl((i_m,j_m), 
(j_m,k_m)\bigr)$ over the set of vertices $(i,j)$, $0 \le i < j \le n$. By 
Lemma \ref{Lm:DetAn} the volume term is nonzero if and only if $(i_1,j_1), 
(j_1, k_1), \dots, (i_m, j_m), (j_m, k_m)$ are edges of a tree; in 
particular, the pairs $(i_1,j_1), (j_1, k_1), \dots, (i_m, j_m), (j_m, 
k_m)$ are all distinct and therefore $\bigl((i_1,j_1), (j_1,k_1)\bigr), 
\dots, \bigl((i_m,j_m), (j_m,k_m)\bigr)$ is indeed a partial pair matching.

By Remark \ref{Rm:IndepNum} and skew symmetry of $u_{ijk}$ one can ensure 
that $i_s < j_s < k_s$ for every $s = 1, \dots, m$. According to 
\eqref{Eq:ScPr} the scalar products in the formula are then equal to 
$-1/2$, and the volume, according to Lemma \ref{Lm:DetAn}, is $\pm 
\sqrt{(n+1_/2^n}$; thus every term is equal to $(-1)^m \sqrt{n+1}/2^n 
\eps(i_1, \dots, k_m) \prod_{s=1}^m u_{i_s j_s k_s}$, where $\eps(i_1, 
\dots, k_m) \bydef \name{sgn} \name{vol}(e_{i_1 j_1}, e_{j_1 k_1}, \dots, 
e_{i_m j_m}, e_{j_m k_m}) = \pm1$ is the sign of the volume.

Let us draw for every term $\prod_{s=1}^m u_{i_s j_s k_s}$ a $3$-graph with 
vertices $0, \dots, n$ and $3$-edges $(i_s j_s k_s). 1 \le s \le m$. The 
sides $(i_s, j_s)$ and $(j_s, k_s)$ of all the edges (recall that $i_s < 
j_s < k_s$) form a tree, which is contractible (homotopy equivalent to a 
point). Since a triangle can be retracted onto a union of its two sides, 
the $3$-graph obtained is contractible, too, and therefore is a $3$-tree.

 \begin{theorem}[\cite{MV}] \label{Th:MV}
The Pfaffian of the principal $n \times n$-submatrix of the matrix $T$ is 
equal to the sum of the terms $\delta(i_1, \dots, k_m) \prod_{s=1}^m u_{i_s 
j_s k_s}$ where the sign $\delta(i_1, \dots, k_m) = \pm1$ is defined as 
follows. Number the $3$-edges from $1$ to $m$ and consider a product of the 
$3$-cycles $(i_1 j_1 k_1) \dots (i_m j_m k_m)$. This product is a cyclic 
permutation $(a_0 \dots a_n)$. The order of the vertices $a_0 \dots a_n$ 
inside the cycle defines a permutation $\sigma$ of $0, \dots, n$; then 
$\delta(i_1, \dots, k_m)$ is the parity of this permutation. 
 \end{theorem}

Another definition of $\delta(i_1, \dots, k_m)$ (also found in \cite{MV}) 
is the following. Number, again, the $3$-edges; so, for every vertex the 
$3$-edges containing it are linearly ordered. This defines an embedding, up 
to a homotopy, of the $3$-graph into a disk such that the vertices lie in 
its boundary, and for every vertex the linear ordering of the $3$-edges 
containing it corresponds to their ordering ``left to right''. Let 
$\sigma(i_1, \dots, k_m) = (c_0, \dots, c_n)$ be the vertices of the graph 
listed counterclockwise. Then $\sigma$ is a permutation of $0, \dots, n$, 
and $\delta(i_1, \dots, k_m) = \name{sgn} \sigma(i_1, \dots, k_m)$.

To derive Theorem \ref{Th:MV} from Theorem \ref{Th:GenDet} one should prove 

 \begin{lemma}
$\delta(i_1, \dots, k_m) = \eps(i_1, \dots, k_m)$
 \end{lemma}

 \begin{proof}
First, the lemma is true for a simplest $3$-tree with the edges $(123), 
(145), \dots, \linebreak (1,n-1,n)$ --- both signs are $+1$. 

Consider now an arbitrary sequence $i_1, \dots, k_m$, that is, an arbitrary 
$3$-tree with numbered edges. Prove first that both $\eps(i_1, \dots, k_m)$ 
and $\delta(i_1, \dots, k_m)$ do not depend on the choice of the numbering 
of the edges. It suffices to show that $\eps$ and $\delta$ stay the same 
when one exchanges the numbers of the two edges, the $s$-th and the 
$(s+1)$-th. So, in $\name{vol}(e_{i_1 j_1}, \dots, e_{j_m,k_m})$ one 
exchanges the vectors $e_{i_s,j_s}, e_{j_s,k_s}$ with $e_{i_{s+1},j_{s+1}}, 
e_{j_{s+1},k_{s+1}}$. Such an exchange is an even permutation, so 
$\eps(i_1, \dots, k_m) = \name{sgn} \name{vol}(e_{i_1 j_1}, \dots, 
e_{j_m,k_m})$ remains the same. Consider now $\delta(i_1, \dots, k_m)$. If 
the $s$-th and the $(s+1)$-th edge have no common vertices, the embedding 
of the $3$-tree described above does not change, and neither does $\delta$. 
If the edges have a common vertex $r$, the subtrees rooted at $r$ and 
containing the edges are swapped. Denote by $a_1. \dots, a_p$ the vertices 
of the first subtree (other than the root) and by $b_1, \dots, b_q$, the 
vertices of the second subtree. Then the fragments $a_1. \dots, a_p$ and 
$b_1, \dots, b_q$ in the permutation $\sigma(i_1, \dots, k_m)$ exchange 
their places. Every $3$-tree contains an odd number of vertices, so both 
$p$ and $q$ are even here, and the exchange is an even permutation, so that 
$\delta(i_1, \dots, k_m) = \name{sgn} \sigma(i_1, \dots, k_m)$ is 
preserved.

Let now $(i_s j_s k_s)$, $i_s < j_s < k_s$, be the $s$-th edge of the 
$3$-tree. Shift the numbers of the vertices cyclically: $i_s \mapsto j_s 
\mapsto k_s \mapsto i_s$ and prove that this will not influence $\eps$ and 
$\delta$. In $\name{vol}(e_{i_1 j_1}, \dots, e_{j_m k_m})$ one replaces 
$e_{i_s j_s}, e_{j_s k_s}$ by $e_{j_s, k_s}, e_{k_s i_s} = -e_{i_s j_s} - 
e_{j_s k_s}$. So the volume stays the same, and so does its sign, 
$\eps(i_1, \dots, k_m)$. The permutation $\sigma(i_1, \dots, k_m)$ is 
multiplied by the $3$-cycle $(i_s j_s k_s)$, which is even, and therefore 
$\delta(i_1, \dots, k_m)$ would not change either.

Take now a dangling edge $(ijk)$, $i < j < k$, of the $3$-tree. As we 
proved, one may suppose that this edge has the number $m$ and hangs on the 
vertex $i$. Delete the edge, obtaining a new tree $(i_1, \dots, k_{m-1})$ 
with $m-1$ edges and $n-2 = 2m-1$ vertices. In $\name{vol}(e_{i_1 j_1}, 
\dots, e_{j_m k_m})$ one deletes the two last vectors; indices of the other 
vectors do not contain $j_m$ and $k_m$. In the corresponding matrix ($\eps$ 
is the sign of its determinant) the two last rows and the $j_m$-th and the 
$k_m$-th column are deleted; so $\eps(i_1, \dots, k_{m-1}) = 
(-1)^{j_m+k_m-1}\eps(i_1, \dots, k_m)$. On the other hand, permutation 
$\sigma(i_1, \dots, k_{m-1})$ is obtained from $\sigma(i_1, \dots, k_m)$ by 
deletion of the numbers $j_m$ and $k_m$; so the number inversions is 
decreased by $j_m+k_m-1$, and $\delta(i_1, \dots, k_{m-1}) = 
(-1)^{j_m+k_m-1}\delta(i_1, \dots, k_m)$. 

Thus, the equality $\eps(i_1, \dots, k_m) = \delta(i_1, \dots, k_m)$ is 
proved by induction on $m$.
 \end{proof}

\section{The $D_n$ and $B_n$ cases}

The $D_n$ root system in the space $V = \Real^n$ contains vectors $e_{ij}^+ 
= (f_i + f_j)/\sqrt{2}$ and $e_{ij}^- = (f_i - f_j)/\sqrt{2}$, for all $1 
\le i < j \le n$; here $f_1, \dots, f_n \in \Real^n$ is the standard 
orthonormal basis. The $B_n$ root system contains the same vectors and also 
the vectors $f_i$ for all $i = 1, \dots, n$. Denote the corresponding 
reflections by $s_{ij}^+$, $s_{ij}^-$ and $s_i$, respectively. The $D_n$ 
root system is a subset of the $B_n$ root system, hence it is enough to 
consider the $B_n$.

To any set $u = \{u_1, \dots, u_k\}$ of $B_n$-roots we associate a graph 
$G(u)$ with the vertices $1, \dots, n$ and the following edges: for any 
$e_{pq}^- \in u$ the graph $G(u)$ has an edge marked ``$-$'' and directed 
from $p$ to $q$; for any $e_{pq}^+ \in u$ it contains an undirected edge 
joining $p$ and $q$ and marked ``$+$''; for any $f_p \in u$ it contains a 
loop attached to the vertex $p$.

All the vectors $e_{pq}^\sigma$ and $f_p$ have unit length; the other 
scalar products are: 
 %*
 \begin{equation}\label{Eq:ScPrDB}
 \begin{aligned}
(e_{i_1j_1}^-, e_{i_2j_2}^-) &\quad\text{is like in 
\eqref{Eq:ScPr}};\\
(e_{i_1j_1}^+, e_{i_2j_2}^+) &=  \begin{cases}
1, &\text{if {
\setlength{\unitlength}{0.075in}
\begin{picture}(5.00,4.85)
\put(2.60,0.00){{\setbox0=\hbox{$+$}\lower\ht0\box0}}
\put(2.06,4.83){{\setbox0=\hbox{$+$}\lower\ht0\box0}}
\special{em:linewidth 0.014in}
\put(5.00,1.33){\special{em:moveto}}
\put(5.00,2.20){\special{em:lineto}}
\put(5.00,2.33){\special{em:moveto}}
\put(4.78,1.88){\special{em:lineto}}
\put(4.46,1.43){\special{em:lineto}}
\put(4.08,1.08){\special{em:lineto}}
\put(3.76,0.88){\special{em:lineto}}
\put(3.45,0.73){\special{em:lineto}}
\put(2.90,0.58){\special{em:lineto}}
\put(2.48,0.55){\special{em:lineto}}
\put(1.90,0.61){\special{em:lineto}}
\put(1.60,0.70){\special{em:lineto}}
\put(1.30,0.83){\special{em:lineto}}
\put(1.13,0.91){\special{em:lineto}}
\put(0.98,1.01){\special{em:lineto}}
\put(0.85,1.11){\special{em:lineto}}
\put(0.75,1.20){\special{em:lineto}}
\put(0.65,1.28){\special{em:lineto}}
\put(0.51,1.41){\special{em:lineto}}
\put(0.41,1.55){\special{em:lineto}}
\put(0.35,1.63){\special{em:lineto}}
\put(0.28,1.71){\special{em:lineto}}
\put(0.25,1.78){\special{em:lineto}}
\put(0.20,1.85){\special{em:lineto}}
\put(0.16,1.91){\special{em:lineto}}
\put(0.13,1.98){\special{em:lineto}}
\put(0.10,2.05){\special{em:lineto}}
\put(0.06,2.11){\special{em:lineto}}
\put(0.00,2.13){\special{em:lineto}}
\put(0.00,2.13){\special{em:moveto}}
\put(0.10,2.33){\special{em:lineto}}
\put(0.20,2.53){\special{em:lineto}}
\put(0.33,2.73){\special{em:lineto}}
\put(0.50,2.93){\special{em:lineto}}
\put(0.60,3.03){\special{em:lineto}}
\put(0.70,3.13){\special{em:lineto}}
\put(0.76,3.20){\special{em:lineto}}
\put(0.85,3.26){\special{em:lineto}}
\put(0.95,3.33){\special{em:lineto}}
\put(1.05,3.40){\special{em:lineto}}
\put(1.15,3.46){\special{em:lineto}}
\put(1.26,3.53){\special{em:lineto}}
\put(1.40,3.60){\special{em:lineto}}
\put(1.56,3.66){\special{em:lineto}}
\put(1.76,3.73){\special{em:lineto}}
\put(2.05,3.80){\special{em:lineto}}
\put(2.51,3.83){\special{em:lineto}}
\put(3.13,3.76){\special{em:lineto}}
\put(3.36,3.70){\special{em:lineto}}
\put(3.55,3.63){\special{em:lineto}}
\put(3.73,3.55){\special{em:lineto}}
\put(3.88,3.46){\special{em:lineto}}
\put(3.98,3.40){\special{em:lineto}}
\put(4.08,3.33){\special{em:lineto}}
\put(4.20,3.25){\special{em:lineto}}
\put(4.28,3.18){\special{em:lineto}}
\put(4.45,3.01){\special{em:lineto}}
\put(4.60,2.85){\special{em:lineto}}
\put(4.73,2.68){\special{em:lineto}}
\put(4.83,2.51){\special{em:lineto}}
\put(4.93,2.35){\special{em:lineto}}
\put(5.00,2.23){\special{em:lineto}}
\end{picture}}},\\
1/2, &\text{if {
\setlength{\unitlength}{0.075in}
\begin{picture}(6.66,3.35)
\put(5.06,0.00){{\setbox0=\hbox{$+$}\lower\ht0\box0}}
\put(1.33,3.33){{\setbox0=\hbox{$+$}\lower\ht0\box0}}
\put(6.66,1.50){{\setbox0=\hbox{$\scriptstyle\bullet$}\kern-.4\wd0\lower.5\ht0\box0}}
\put(3.33,1.50){{\setbox0=\hbox{$\scriptstyle\bullet$}\kern-.4\wd0\lower.5\ht0\box0}}
\put(0.00,1.50){{\setbox0=\hbox{$\scriptstyle\bullet$}\kern-.4\wd0\lower.5\ht0\box0}}
\special{em:linewidth 0.014in}
\put(3.33,1.50){\special{em:moveto}}
\put(6.53,1.50){\special{em:lineto}}
\put(0.00,1.50){\special{em:moveto}}
\put(3.33,1.50){\special{em:lineto}}
\end{picture}
}}, \\
0, &\text{otherwise (i.e.\ if the edges have no common vertices)}.
 \end{cases}\\
(e_{i_1j_1}^+, e_{i_2j_2}^-) &=  \begin{cases}
1/2, &\text{if {
\setlength{\unitlength}{0.075in}
\begin{picture}(6.66,3.35)
\special{em:linewidth 0.014in}
\put(6.63,1.50){\special{em:moveto}}
\put(5.90,1.13){\special{em:lineto}}
\put(5.80,1.86){\special{em:moveto}}
\put(6.63,1.50){\special{em:lineto}}
\put(5.06,0.00){{\setbox0=\hbox{$-$}\lower\ht0\box0}}
\put(1.33,3.33){{\setbox0=\hbox{$+$}\lower\ht0\box0}}
\put(6.66,1.50){{\setbox0=\hbox{$\scriptstyle\bullet$}\kern-.4\wd0\lower.5\ht0\box0}}
\put(3.33,1.50){{\setbox0=\hbox{$\scriptstyle\bullet$}\kern-.4\wd0\lower.5\ht0\box0}}
\put(0.00,1.50){{\setbox0=\hbox{$\scriptstyle\bullet$}\kern-.4\wd0\lower.5\ht0\box0}}
\put(3.33,1.50){\special{em:moveto}}
\put(6.53,1.50){\special{em:lineto}}
\put(0.00,1.50){\special{em:moveto}}
\put(3.33,1.50){\special{em:lineto}}
\end{picture}
}},\\
&\text{but the final point does not},\\ 
-1/2, &\text{if {
\setlength{\unitlength}{0.075in}
\begin{picture}(6.53,3.35)
\special{em:linewidth 0.014in}
\put(3.30,1.50){\special{em:moveto}}
\put(4.23,1.10){\special{em:lineto}}
\put(4.20,2.00){\special{em:moveto}}
\put(3.30,1.50){\special{em:lineto}}
\put(5.06,0.00){{\setbox0=\hbox{$-$}\lower\ht0\box0}}
\put(1.33,3.33){{\setbox0=\hbox{$+$}\lower\ht0\box0}}
\put(3.33,1.50){{\setbox0=\hbox{$\scriptstyle\bullet$}\kern-.4\wd0\lower.5\ht0\box0}}
\put(0.00,1.50){{\setbox0=\hbox{$\scriptstyle\bullet$}\kern-.4\wd0\lower.5\ht0\box0}}
\put(3.33,1.50){\special{em:moveto}}
\put(6.53,1.50){\special{em:lineto}}
\put(0.00,1.50){\special{em:moveto}}
\put(3.33,1.50){\special{em:lineto}}
\end{picture}
}},\\ 
0, &\text{otherwise}.
 \end{cases}\\
(e_{ij}^+,f_p) &= \begin{cases}
1/\sqrt{2} &\text{if {
\setlength{\unitlength}{0.075in}
\begin{picture}(5.75,2.23)
\special{em:linewidth 0.014in}
\put(3.23,0.51){\special{em:moveto}}
\put(3.43,0.60){\special{em:lineto}}
\put(3.63,0.68){\special{em:lineto}}
\put(3.81,0.75){\special{em:lineto}}
\put(3.98,0.81){\special{em:lineto}}
\put(4.15,0.86){\special{em:lineto}}
\put(4.30,0.90){\special{em:lineto}}
\put(4.45,0.93){\special{em:lineto}}
\put(4.60,0.96){\special{em:lineto}}
\put(4.73,0.98){\special{em:lineto}}
\put(4.85,0.98){\special{em:lineto}}
\put(4.95,0.98){\special{em:lineto}}
\put(5.06,0.96){\special{em:lineto}}
\put(5.16,0.96){\special{em:lineto}}
\put(5.26,0.93){\special{em:lineto}}
\put(5.35,0.91){\special{em:lineto}}
\put(5.41,0.88){\special{em:lineto}}
\put(5.48,0.85){\special{em:lineto}}
\put(5.55,0.81){\special{em:lineto}}
\put(5.60,0.76){\special{em:lineto}}
\put(5.65,0.73){\special{em:lineto}}
\put(5.68,0.68){\special{em:lineto}}
\put(5.71,0.63){\special{em:lineto}}
\put(5.73,0.58){\special{em:lineto}}
\put(5.75,0.53){\special{em:lineto}}
\put(5.75,0.48){\special{em:lineto}}
\put(5.73,0.43){\special{em:lineto}}
\put(5.73,0.38){\special{em:lineto}}
\put(5.70,0.33){\special{em:lineto}}
\put(5.68,0.28){\special{em:lineto}}
\put(5.65,0.23){\special{em:lineto}}
\put(5.60,0.20){\special{em:lineto}}
\put(5.55,0.15){\special{em:lineto}}
\put(5.48,0.11){\special{em:lineto}}
\put(5.41,0.08){\special{em:lineto}}
\put(5.33,0.06){\special{em:lineto}}
\put(5.25,0.03){\special{em:lineto}}
\put(5.15,0.01){\special{em:lineto}}
\put(5.05,0.00){\special{em:lineto}}
\put(4.95,0.00){\special{em:lineto}}
\put(4.83,0.00){\special{em:lineto}}
\put(4.71,0.01){\special{em:lineto}}
\put(4.58,0.03){\special{em:lineto}}
\put(4.43,0.05){\special{em:lineto}}
\put(4.28,0.08){\special{em:lineto}}
\put(4.13,0.13){\special{em:lineto}}
\put(3.96,0.18){\special{em:lineto}}
\put(3.78,0.25){\special{em:lineto}}
\put(3.60,0.31){\special{em:lineto}}
\put(3.40,0.41){\special{em:lineto}}
\put(3.20,0.51){\special{em:lineto}}
\put(0.76,2.21){{\setbox0=\hbox{$+$}\lower\ht0\box0}}
\put(0.00,0.51){{\setbox0=\hbox{$\scriptstyle\bullet$}\kern-.4\wd0\lower.5\ht0\box0}}
\put(0.00,0.51){\special{em:moveto}}
\put(3.33,0.51){\special{em:lineto}}
\end{picture}}}\\
0 &\text{otherwise}
 \end{cases}\\
(e_{ij}^-,f_p) &= \begin{cases}
-1/\sqrt{2} &\text{if {
\setlength{\unitlength}{0.075in}
\begin{picture}(5.75,2.23)
\special{em:linewidth 0.014in}
\put(3.30,0.51){\special{em:moveto}}
\put(2.53,0.08){\special{em:lineto}}
\put(2.43,0.95){\special{em:moveto}}
\put(3.30,0.51){\special{em:lineto}}
\put(3.23,0.51){\special{em:moveto}}
\put(3.43,0.60){\special{em:lineto}}
\put(3.63,0.68){\special{em:lineto}}
\put(3.81,0.75){\special{em:lineto}}
\put(3.98,0.81){\special{em:lineto}}
\put(4.15,0.86){\special{em:lineto}}
\put(4.30,0.90){\special{em:lineto}}
\put(4.45,0.93){\special{em:lineto}}
\put(4.60,0.96){\special{em:lineto}}
\put(4.73,0.98){\special{em:lineto}}
\put(4.85,0.98){\special{em:lineto}}
\put(4.95,0.98){\special{em:lineto}}
\put(5.06,0.96){\special{em:lineto}}
\put(5.16,0.96){\special{em:lineto}}
\put(5.26,0.93){\special{em:lineto}}
\put(5.35,0.91){\special{em:lineto}}
\put(5.41,0.88){\special{em:lineto}}
\put(5.48,0.85){\special{em:lineto}}
\put(5.55,0.81){\special{em:lineto}}
\put(5.60,0.76){\special{em:lineto}}
\put(5.65,0.73){\special{em:lineto}}
\put(5.68,0.68){\special{em:lineto}}
\put(5.71,0.63){\special{em:lineto}}
\put(5.73,0.58){\special{em:lineto}}
\put(5.75,0.53){\special{em:lineto}}
\put(5.75,0.48){\special{em:lineto}}
\put(5.73,0.43){\special{em:lineto}}
\put(5.73,0.38){\special{em:lineto}}
\put(5.70,0.33){\special{em:lineto}}
\put(5.68,0.28){\special{em:lineto}}
\put(5.65,0.23){\special{em:lineto}}
\put(5.60,0.20){\special{em:lineto}}
\put(5.55,0.15){\special{em:lineto}}
\put(5.48,0.11){\special{em:lineto}}
\put(5.41,0.08){\special{em:lineto}}
\put(5.33,0.06){\special{em:lineto}}
\put(5.25,0.03){\special{em:lineto}}
\put(5.15,0.01){\special{em:lineto}}
\put(5.05,0.00){\special{em:lineto}}
\put(4.95,0.00){\special{em:lineto}}
\put(4.83,0.00){\special{em:lineto}}
\put(4.71,0.01){\special{em:lineto}}
\put(4.58,0.03){\special{em:lineto}}
\put(4.43,0.05){\special{em:lineto}}
\put(4.28,0.08){\special{em:lineto}}
\put(4.13,0.13){\special{em:lineto}}
\put(3.96,0.18){\special{em:lineto}}
\put(3.78,0.25){\special{em:lineto}}
\put(3.60,0.31){\special{em:lineto}}
\put(3.40,0.41){\special{em:lineto}}
\put(3.20,0.51){\special{em:lineto}}
\put(0.76,2.21){{\setbox0=\hbox{$-$}\lower\ht0\box0}}
\put(0.00,0.51){{\setbox0=\hbox{$\scriptstyle\bullet$}\kern-.4\wd0\lower.5\ht0\box0}}
\put(0.00,0.51){\special{em:moveto}}
\put(3.33,0.51){\special{em:lineto}}
\end{picture}
}}\\ 
1/\sqrt{2} &\text{if {
\setlength{\unitlength}{0.075in}
\begin{picture}(5.81,2.23)
\special{em:linewidth 0.014in}
\put(0.00,0.51){\special{em:moveto}}
\put(0.86,0.18){\special{em:lineto}}
\put(0.80,0.91){\special{em:moveto}}
\put(0.00,0.51){\special{em:lineto}}
\put(3.30,0.51){\special{em:moveto}}
\put(3.50,0.60){\special{em:lineto}}
\put(3.70,0.68){\special{em:lineto}}
\put(3.88,0.75){\special{em:lineto}}
\put(4.05,0.81){\special{em:lineto}}
\put(4.21,0.86){\special{em:lineto}}
\put(4.36,0.90){\special{em:lineto}}
\put(4.51,0.93){\special{em:lineto}}
\put(4.66,0.96){\special{em:lineto}}
\put(4.80,0.98){\special{em:lineto}}
\put(4.91,0.98){\special{em:lineto}}
\put(5.01,0.98){\special{em:lineto}}
\put(5.13,0.96){\special{em:lineto}}
\put(5.23,0.96){\special{em:lineto}}
\put(5.33,0.93){\special{em:lineto}}
\put(5.41,0.91){\special{em:lineto}}
\put(5.48,0.88){\special{em:lineto}}
\put(5.55,0.85){\special{em:lineto}}
\put(5.61,0.81){\special{em:lineto}}
\put(5.66,0.76){\special{em:lineto}}
\put(5.71,0.73){\special{em:lineto}}
\put(5.75,0.68){\special{em:lineto}}
\put(5.78,0.63){\special{em:lineto}}
\put(5.80,0.58){\special{em:lineto}}
\put(5.81,0.53){\special{em:lineto}}
\put(5.81,0.48){\special{em:lineto}}
\put(5.80,0.43){\special{em:lineto}}
\put(5.80,0.38){\special{em:lineto}}
\put(5.76,0.33){\special{em:lineto}}
\put(5.75,0.28){\special{em:lineto}}
\put(5.71,0.23){\special{em:lineto}}
\put(5.66,0.20){\special{em:lineto}}
\put(5.61,0.15){\special{em:lineto}}
\put(5.55,0.11){\special{em:lineto}}
\put(5.48,0.08){\special{em:lineto}}
\put(5.40,0.06){\special{em:lineto}}
\put(5.31,0.03){\special{em:lineto}}
\put(5.21,0.01){\special{em:lineto}}
\put(5.11,0.00){\special{em:lineto}}
\put(5.01,0.00){\special{em:lineto}}
\put(4.90,0.00){\special{em:lineto}}
\put(4.78,0.01){\special{em:lineto}}
\put(4.65,0.03){\special{em:lineto}}
\put(4.50,0.05){\special{em:lineto}}
\put(4.35,0.08){\special{em:lineto}}
\put(4.20,0.13){\special{em:lineto}}
\put(4.03,0.18){\special{em:lineto}}
\put(3.85,0.25){\special{em:lineto}}
\put(3.66,0.31){\special{em:lineto}}
\put(3.46,0.41){\special{em:lineto}}
\put(3.26,0.51){\special{em:lineto}}
\put(0.83,2.21){{\setbox0=\hbox{$-$}\lower\ht0\box0}}
\put(0.06,0.51){{\setbox0=\hbox{$\scriptstyle\bullet$}\kern-.4\wd0\lower.5\ht0\box0}}
\put(0.06,0.51){\special{em:moveto}}
\put(3.40,0.51){\special{em:lineto}}
\end{picture}
}}\\ 
0 &\text{otherwise}
 \end{cases}
 \end{aligned}
 \end{equation}
 %*

Consider a graph with $n$ vertices and $n$ edges, some of them directed 
(the ``$-$''-edges), some not (the ``$+$''-edges), loops allowed (and not 
directed, indeed). Such a graph will be called $B$-basic if any connected 
component of it contains exactly one cycle, and this cycle either is a loop 
or has an odd number of ``$+$''-edges in it. The following lemma explains 
the term:

 \begin{lemma}\label{Lm:DB}
Let $u = \{u_1, \dots, u_n\} \in \Real^n$ be a system of $B_n$-roots. The 
volume $\name{vol}(u_1, \dots, u_n)$ is nonzero (that is, $u$ is a basis in 
$\Real^n$) if and only if the graph $G(u)$ is $B$-basic. If $G(u)$ is 
$B$-basic, then the volume is equal to $\pm 2^{d-\ell/2-n/2}$ 
where $d$ is the number of connected components in $G(u)$ and $\ell$ is the 
number of loops in it.
 \end{lemma}

 \begin{proof}
Suppose first that $G(u)$ contains no loops (that is, $u$ is a system of 
$D_n$-roots).

Let the edges $(i_1, i_2), \dots, (i_m, i_1)$, marked $\sigma_1, \dots, 
\sigma_m$ (where $\sigma_i = +$ or $-$), form a cycle. Denote by $\tau_p$ 
the number of $+$s among $\sigma_1, \dots, \sigma_p$. Then $e_{i_1 
i_2}^{\sigma_1} + (-1)^{\tau_1} e_{i_2 i_3}^{\sigma_2} + \dots + 
(-1)^{\tau_{m-1}} e_{i_m i_1}^{\sigma_m} = f_{i_1}(1 + (-1)^{\tau_m + 1})$. 
If $\tau_m$ is even (that is, the cycle contains an even number of 
``$+$''-edges), then the sum is zero and the vectors are linearly 
dependent, hence $\name{vol}(u) = 0$. 

Let a connected component of $G(u)$ contain two cycles, $(i_1, \dots, i_p)$ 
and $(j_1, \dots, j_q)$ joined by a path $k_1, k_2, \dots, k_r$ where $k_1 
= i_1$, $k_r = j_1$. If either cycle contains an even number of 
``$+$''-edges, $u$ is proved to be lineary dependent. Suppose that each 
cycle contain an odd number of ``$+$''-edges. Then the linear hull of $u$ 
contains vectors $2f_{i_1}$ and $2f_{j_1}$, as well as the vector $e_{k_1 
k_2}^{\sigma_1} + (-1)^{\tau_1} e_{k_2 k_3}^{\sigma_2} + \dots + 
(-1)^{\tau_{r-2}} e_{k_{r-1} k_r}^{\sigma_r} = f_{i_1} + (-1)^{\tau_{r-1}} 
f_{j_1}$ (notation as above) --- hence, $u$ is linearly dependent, too. So, 
if a connected component of a $G(u)$ contains two (or more) cycles, $u$ is 
linearly dependent and $\name{vol}(u) = 0$.

The graph $G(u)$ has $n$ vetices and $n$ edges. So if every connected 
component of $G(u)$ contains not more than one cycle, it contains exactly 
one. Thus, the first part of the lemma is proved for the $D_n$ case (no 
loops).

If $f_p \in u$ then $G(u)$ contains a loop. Add then a new dimension, and 
replace $f_p$ by $e_{p,n+1}^-$ and $e_{p,n+1}^+$, obtaining a new system 
$u'$. Since $f_p = \frac{1}{\sqrt2}(e_{p,n+1}^- + e_{p,n+1}^+)$, the system 
$u'$ is linearly dependent if and only if $u$ is. Thus, the first part of 
the lemma for the $B_n$ case follows from the same statement for the $D_n$.

Suppose now that $G(u)$ is connected and $B$-basic. Let $(ij)$ and $(jk)$ 
be two edges of $G(u)$ sharing a vertex $j$. Changing the signs of the 
vectors if necessary (this will affect the sign of $\name{vol}(u)$ but not 
its absolute value) we may suppose that the vectors corresponding to these 
edges are $a = e_{ij}^{\sigma_1} \in u$ and $b = e_{jk}^{\sigma_2} \in u$. 
Replace then the vector $b$ by $c = \pm b \pm a$ (the signs depend on 
$\sigma_1$ and $\sigma_2$ and are easily determined) one arrives to a 
system $u'$ containing $a$ and $c = e_{ik}^\sigma$ ($\sigma$ also depends 
on $\sigma_1$ and $\sigma_2$); it is clear that $\name{vol}(u') = \pm 
\name{vol}(u)$ and that $G(u')$ is still connected and $B$-basic. Applying 
this transformation several times one can replace $G(u)$ with a 
connected $B$-basic graph containing either a loop $f_i$ or a pair of edges 
$e_{ij_1}^+$, $e_{i j_1}^-$ (the exponents must be different because the 
graph is $B$-basic), and some edges $e_{i j_p}^{\sigma_p}$, $p = 2, \dots, 
n$, all $j_p$ being distinct. In the first case the volume is 
 %*
 \begin{equation*}
\det \frac{1}{\sqrt2}\left(\begin{array}{cccc} \sqrt2 & 1 & 1 & 
\dots\\ 0 & \pm1 & 0 & \dots\\ 0 & 0 & \pm1 & \dots \end{array} \right) = 
\pm 2^{1-1/2-n/2},
 \end{equation*}
 %*
and in the second case it is 
 %*
 \begin{equation*}
\det \frac{1}{\sqrt2}\left(\begin{array}{ccccc} 1 & 1 & 1 & 1 & \dots\\ 1 & 
-1 & 0 & 0 & \dots\\ 0 & 0 & \pm1 & 0 & \dots\\ 0 & 0 & 0 & \pm1 & \dots 
\end{array} \right) = \pm 2^{1-n/2},
 \end{equation*}
 %*
(easily checked using column expansion).

In the general case a $B$-basic graph $G(u)$ is not connected but every its 
connected component contains exactly one cycle. If two edges have no common 
vertices, then the corresponding vectors are orthogonal (see 
\eqref{Eq:ScPrDB}), so the volume is the product of volumes (of appropriate 
dimension) corresponding to connected components; thus the volume is 
$2^{d-\ell/2-n/2}$ where $d$ is the number of connected components of 
$G(u)$ and $\ell$, the number of loops in it.
 \end{proof}

Apply now Theorem \ref{Th:GenDet} for $k=1$ to the $B_n$ case. Denote by 
$T_w$ the matrix of the operator $P_w^{(1)} = \sum_{1 \le i < j \le n} 
(w_{ij}^+ (I - s_{ij}^+) + w_{ij}^- (I - s_{ij}^-)) + \sum_{1 \le i \le n} 
w_i (I - s_i)$ in the standard basis $f_1, \dots, f_n$. The matrix $T_w$ is 
symmetric; one has $(T_w)_{ij} = -w_{ij}^- + w_{ij}^+$ for $i \ne j$ , and 
its $(i,i)$-th entry is $(T_w)_{ii} = \sum_{j \ne i} (w_{ij}^+ + w_{ij}^-) 
+ 2w_i$; here we assume that $w_{ij}^\sigma = w_{ji}^\sigma$ where $\sigma 
= +$ or $-$. Combining this with Theorem \ref{Th:GenDet} for $k=1$ and 
Lemma \ref{Lm:DB} we obtain the following $B$-version of the matrix-tree 
theorem:

 \begin{theorem}
The determinant of the matrix $T_w$ described above is equal to the sum of 
weights of all the $B$-basic graphs on the vertex set $1, \dots, n$, the 
weight of each graph multiplied by $2^{2d}$ where $d$ is the number of its 
connected components. A weight of a $B$-basic graph is defined as the 
product of all $w_{ij}^+$ where $(ij)$ is a ``$+$''-edge, times the product 
of all $w_{ij}^-$ where $(ij)$ is a ``$-$''-edge, times the product of all 
the $w_i$ where $i$ is a vertex to which a loop is attached. 
 \end{theorem}

 \begin{proof}
Follows immediately from Theorem \ref{Th:GenDet} and Lemma \ref{Lm:DB}; 
note that the factor $2$ at the $w_i$ compensates for $2^{-\ell/2}$ term in 
Lemma \ref{Lm:DB}).
 \end{proof}

An analog of the Massbaum--Vaintrob theorem for the $B_n$ case can also be 
proved using Theorem \ref{Th:GenDet} for $k=2$. The formula obtained, 
though, involves some messy summation over the set of the $3$-graph with 
possibly singular $3$-edges and with additional structure on it; we do not 
formulate the result here.

\subsection*{Acknowledgments}

The first named author was generously supported by the RFBR grants 
08-01-00110-  ``Geometry and combinatorics of mapping spaces for real and 
complex curves'' and NSh-8462.2010.1 ``Vladimir Arnold's school'', by the 
HSE Scientific foundation grant 10-01-0029 ``Graph embeddings and 
integrable hierarchies'', by the joint RFBR ans HSE grant 09-01-12185-ofi-m 
``Combinatorial aspects of integrable models in mathematical physics'', HSE 
Laboratory of mathematical research TZ 62.0 project ``Math research in 
low-dimensional topology, algebraic geometry and representation theory'', 
and by the FTsP-Kadry grant ``New methods of research of integrable systems 
and moduli spaces in geometry, topology and mathematical physics''.

The authors are indebted to all the participants of Serge Lando's seminar 
on characteristic classes (Independent University of Moscow and Higher 
School of Economics), and especially to Maxim Kazarian, for valuable 
discussions.

\end{document}